\documentclass[11pt]{article}

\usepackage[dvips]{epsfig} 
\usepackage{amssymb,amsmath,amsthm,mathrsfs} 
\usepackage{graphicx}
\usepackage{lineno}
\usepackage{multirow}

\usepackage{url}
\usepackage{enumerate}
\usepackage{colordvi,color,pspicture}
\usepackage{psfrag}
\usepackage{pgf,tikz}
\usepackage{mathrsfs}
\usetikzlibrary{arrows}
\usetikzlibrary{decorations.markings}

\usepackage[utf8]{inputenc}

\newcommand{\hp}{\textsc{Hypergraph Transversal~}}

\newtheorem{theorem}{Theorem}[section]
\newtheorem{corollary}[theorem]{Corollary}
\newtheorem{lemma}[theorem]{Lemma}
\newtheorem{proposition}[theorem]{Proposition}
\newtheorem{observation}[theorem]{Observation}

\newtheorem{remark}[theorem]{Remark}

\theoremstyle{definition}

\begin{document}
\title{Well-Totally-Dominated Graphs}
\author{
Selim Bahadır\thanks{Department of Mathematics, Ankara Yıldırım Beyazıt University, Turkey. E-mail: sbahadir@ybu.edu.tr}
\and
Tınaz Ekim \thanks{Department of Industrial Engineering, Boğaziçi University, Turkey. E-mail: tinaz.ekim@boun.edu.tr}
\and
Didem Gözüpek\thanks{Department of Computer Engineering, Gebze Technical University, Kocaeli, Turkey. E-mail: didem.gozupek@gtu.edu.tr}\and
}
\date{\today}

\maketitle
\pagenumbering{roman} \setcounter{page}{1}

\pagenumbering{arabic} \setcounter{page}{1}

\begin{abstract}
\noindent
A subset of vertices in a graph is called a total dominating set if every vertex of the graph is adjacent to at least one vertex of this set. A total dominating set is called minimal if it does not properly contain another total dominating set. In this paper, we study graphs whose all minimal total dominating sets have the same size, referred to as well-totally-dominated (WTD) graphs. We first show that WTD graphs with bounded total domination number can be recognized  in polynomial time. Then we focus on WTD graphs with total domination number two. In this case, we characterize triangle-free WTD graphs and WTD graphs with packing number two, and we show that there are only finitely many planar WTD graphs with minimum degree at least three. Lastly, we show that if the minimum degree is at least three then the girth of a WTD graph is at most 12. We conclude with several open  questions.
\end{abstract}

\noindent
{\it Keywords: Total domination, well-totally-dominated graphs, minimal total dominating sets.} \\


\section{Introduction}
\label{sec:intro}

Total domination in graphs has been extensively studied in the literature (see \cite{HeYe13}) and has numerous applications. For instance, consider a computer network where a core group of file servers has the ability to communicate directly with every computer outside the core group. Moreover, each file server is directly linked to at least one other backup file server where duplicate information is stored. This core group of servers corresponds to a total dominating set in the graph representing the computer network. Another application area is a specific committee selection mechanism such that every non-member of the committee knows at least one member of the committee and every member of the committee knows at least one other member of the committee to avoid feelings of isolation and thus enhance cooperation (see \cite{haynes:1998}). 

Let $G$ be a graph with no isolated vertices.
A subset $S$ of $V(G)$ is called a \emph{total dominating set} (TDS) of $G$ if every vertex in $G$ is adjacent to at least one element in $S$.
A total dominating set is \emph{minimal} if it contains no other TDS of $G$. The minimum size of a total dominating set of a graph $G$ is called the \emph{total domination number} and denoted by $\gamma_{t}(G)$, while the maximum size of a minimal total dominating set is called the \emph{upper total domination number} and denoted by $\Gamma_{t}(G)$. $G$ is called \emph{well-totally-dominated} (WTD) if every minimal TDS of $G$ is of the same size, that is, $\gamma_t(G)=\Gamma_t(G)$.
WTD graphs with $\gamma_t=k$ are denoted by WTD($k$).

Given a graph, computing its total domination number and its upper total domination number are NP-hard in general \cite{pfaff:1984, fang:2004} and already NP-hard even in specific graph classes such as bipartite graphs, comparability graphs and claw-free graphs \cite{HeYe13}. One way to deal with such a problem is to consider ``trivial" instances where these two paramaters have the same value. Examples of graph classes defined in this way in the literature include well-covered graphs (whose all maximal independent sets have the same size), well-dominated graphs (whose all minimal dominating sets have the same size), and equimatchable graphs (whose all maximal matchings have the same size). Structural properties of each one of these graph classes have been studied extensively in the literature.  In this paper, we take the same approach for the total dominating sets. Works on total domination in the literature mostly focused on the relation of the total domination number with other graph parameters and characterized graphs with total domination number being equal to an upper bound (e.g. \cite{CDH80, BCV20}). 
Inequalities relating the total domination number to other domination parameters and characterization of graphs that tightly attain these bounds have also been studied (see \cite{HLX10, BaGo18}).

Clearly, if the total domination number and the upper total domination number are polynomial time solvable for a given class of graphs, then the recognition of WTD graphs belonging to this class of graphs is polynomial. However, the complexity of recognizing WTD graphs in general is unknown. In such a situation, a classical approach consists in studying the structure of WTD graphs in restricted graph classes and providing structural characterizations along with efficient recognition algorithms whenever possible.

WTD graphs were initially introduced in \cite{HaRa97}, where WTD cycles and paths are characterized and several constructions of WTD trees are given. They also proved that a WTD graph with minimum degree at least two has girth at most 14. The work in \cite{FFV09} focused on the composition and decomposition of WTD trees and proved that any WTD tree can be constructed from a family of three small trees. To the best of our knowledge, \cite{HaRa97} and \cite{FFV09} are the only work on WTD graphs.
A graph class resembling WTD graphs is \emph{well-dominated graphs}, which are graphs whose minimal dominating sets have the same size. 
It is known that well-dominated graphs form a proper subset of well-covered graphs \cite{FHN88}. 
We note that well-covered graphs are graphs whose maximal independent sets have the same size and there is a rich literature about them (see \cite{Plu93, Har99}). 
Well-dominated graphs were introduced by Finbow et.al. \cite{FHN88}, who provided a characterization of bipartite well-dominated graphs and well-dominated graphs with girth at least 5. 
Characterizations of these graphs within other graph classes were also obtained \cite{ToVo90, GKS11, FiBo15, LeTa17}. 
Although their definitions resemble each other, there is not a containment relationship between WTD graphs and well-dominated graphs. 
For instance, a cycle on six vertices is WTD but not well-dominated, whereas the graph $T_{10}$ described in \cite{LeTa17} is well-dominated but not WTD.

It follows from the previous studies on WTD graphs that we do not know much about their structure. 
In this paper, we investigate the study of WTD graphs from a structural point of view. We first study WTD graphs with bounded total domination number. We prove in Section \ref{sec:recWTD} that the recognition of WTD graphs with total domination number $k$ is solvable in polynomial time for every positive integer $k$. We then focus on WTD graphs with total domination number $2$, referred to as WTD(2) graphs in Section \ref{sec:WTDtwo}. We characterize triangle-free WTD(2) graphs and WTD(2) graphs with packing number 2 (or equivalently of diameter 3). We also show that there is a finite number of planar WTD(2) graphs with minimum degree at least 3. Subsequently, we study the girth of WTD graphs in Section \ref{sec:girthWTD}. In particular, building on a result in \cite{HaRa97}, we prove that WTD graphs with minimum degree at least three has girth at most 12. Finally, we discuss several open research directions.

\section{WTD Graphs with Bounded Total Domination Number}\label{sec:recWTD}

Recall that the complexity of recognizing WTD graphs is unknown. In this section, we show that for any positive integer $k$, WTD($k$) graphs can be recognized in polynomial time. 
To this end, we will use an equivalent description of WTD($k$) graphs using transversal hypergraphs. 
Let us first introduce necessary definitions. 
A \emph{transversal} (or \emph{hitting set}) of a hypergraph $H = (X, E)$ is a set $T\subseteq X$ that has nonempty intersection with every edge of $H$.
A transversal of a collection of sets is a transversal of the hypergraph whose hyperedges are the given collection.
A transversal $T$ is called minimal if no proper subset of $T$ is a transversal.
The \emph{transversal hypergraph} of $H=(X,E)$ is the hypergraph $H^*=(X, F)$ whose edge set $F$ consists of all minimal transversals of $H$.

Let $G$ be a graph with no isolated vertex.
Let $H_G$ be the hypergraph whose vertex set is $V(G)$ and hyperedges are open neighborhoods of the vertices of $G$.
Let also $MTDS(G)$ denote the set of all minimal total dominating sets of $G$.

Let $T$ be a hyperedge of $H_G^*$, that is a minimal transversal of the set of open neighborhoods of $G$. This means that $T$ contains a neighbor of every vertex in $G$, thus it is a total dominating set. By minimality of the transversal $T$, it is also a minimal total dominating set of $G$. Conversely, every vertex of $G$ is adjacent to at least one element in a total dominating set TDS. Thus every TDS contains at least one vertex from every open neighborhood of the vertices of $G$

\begin{lemma}\label{lem:hg}
$MTDS(G)$ consists of hyperedges of the transversal hypergraph of $H_G$.
\end{lemma}
\begin{proof}
Let $T$ be a hyperedge of $H_G^*$, that is a minimal transversal of the set of open neighborhoods of $G$. 
This means that $T$ contains a neighbor of every vertex in $G$, thus it is a total dominating set. 
By minimality of the transversal $T$, it is also a minimal total dominating set of $G$. 
Conversely, let $S$ be a minimal total dominating set of $G$.
Then, every vertex in $G$ is adjacent to at least one vertex in $S$, that is, $S$ has a nonempty intersection with every open neighborhood in $G$.
Therefore, $S$ is a transversal of the hypergraph $H_G$ and minimality of $S$ implies that it is a minimal transversal, thus $S$ is a hyperedge of $H_G^*$.
\end{proof}

\begin{proposition}\label{prop:transmtdsg}
Let $G$ be a graph.
Then, for any minimal transversal $T$ of $MTDS(G)$, there exists a vertex $v$ in $G$ such that $N(v)=T$.
\end{proposition}
\begin{proof}
Let $MTDS(G)=\{A_1,\dots,A_m\}$.
Since $T$ has nonempty intersection with each $A_i$,
$V(G)\backslash T$ contains none of the minimal total dominating sets $A_i$s.
Therefore, $V(G)\backslash T$ is not a TDS of $G$,
and hence there exists at least one vertex $v\in V(G)$ such that $N(v)\cap V(G)\backslash T=\emptyset$. 
Thus, we see that $N(v)\subseteq T$.
Suppose that $N(v)\neq T$.
Then $T\backslash N(v)\neq \emptyset$ and let $u\in T\backslash N(v)$.
Since $T$ is a minimal transversal, $T\backslash \{u\}$ is disjoint with at least one of $A_i$s, say $A_1$.
As $u\in T\backslash N(v)$, we have $N(v) \subseteq T\backslash \{u\}$,
and hence $N(v)\cap A_1=\emptyset$, that is, $v$ is not dominated by $A_1$, contradiction.
Therefore, $N(v)=T$.
\end{proof}

A hypergraph $H$ is said to be \emph{Sperner} if no hyperedge of $H$ contains another hyperedge.
The following result shows that any finite collection of finite sets which forms a Sperner hypergraph corresponds to the set of all minimal total dominating sets of a graph.

\begin{proposition}\label{prop:spernermtdsg}
Let $H$ be a Sperner hypergraph.
Then there exists a graph $G$ such that $E(H)=MTDS(G)$.
\end{proposition}
\begin{proof}
Let $E(H)=\{A_1,\dots, A_m \}$ and $A=\cup _{i=1}^m A_i$.
Consider a graph with vertex set $A$ and draw edges between its vertices such that each vertex is adjacent to at least one vertex in $A_i$ for all $i=1,\dots,m$ (for example, draw all possible edges).
Then, in accordance with Proposition \ref{prop:transmtdsg}, for each minimal transversal $T$ of $H$, add a vertex $v_T$ to the graph such that $N(v_T)=T$.
Let $G$ be the resulting graph.

We first show that each $A_i$ is a TDS of $G$.
By construction, every vertex of $A$ is adjacent to at least one vertex in $A_i$.
Moreover, for every minimal transversal $T$ of $A_1,\dots,A_m$ we have $T\cap A_i\neq \emptyset$,
and hence, each $v_T$ is dominated by $A_i$.
Therefore, $A_i$ is a TDS for $i=1,\dots,m$.

We next show that every TDS of $G$ contains at least one of $A_i$s.
Let $S$ be a TDS of $G$ and suppose that $A_i \nsubseteq S$ for $i=1,\dots,m$.
Then, $V(G)\backslash S$ is a transversal of $A_1,\dots,A_m$, and hence,
there exists a minimal transversal $T$ of $A_1,\dots,A_m$ such that $T\subseteq V(G)\backslash S$.
On the other hand, we have $N(v_T)=T$ and thus, we get $N(v_T)\cap S=\emptyset$, which contradicts with $S$ being a TDS of $G$.

Consequently, a set other than $A_1,\dots,A_m$ can not be a minimal TDS of $G$.
We finally show that each $A_i$ is a minimal TDS of $G$.
Suppose that $A_i$ is not minimal for some $i$.
Then, $A_i\backslash \{x\}$ is still a TDS of $G$ for some $x\in A_i$, and therefore,
$A_j \subseteq A_i\backslash \{x\} $ for some $j$, which implies $A_j \subseteq A_i$, contradiction to $H$ being Sperner.
Therefore, minimal TDSs of $G$ are exactly $A_1,\dots, A_m$.
\end{proof}

\begin{remark}
One can extend $G$ to another graph whose minimal TDSs are $A_1,\dots, A_m$ as follows:
Let $G'$ be a graph disjoint from $G$.
Draw edges between the vertices of $G'$ and $A$ in such a way that every vertex of $G'$ is adjacent to at least one vertex of $A_i$ for $i=1,\dots,m$.
By following the same arguments, it is easy to check that minimal TDSs of the resulting graph are $A_1,\dots,A_m$.
\end{remark}

Notice that any finite collection consisting of distinct sets of size $k$
corresponds to a Sperner hypergraph and therefore, 
Proposition \ref{prop:spernermtdsg} implies the following result.

\begin{corollary}\label{cor:infwtdk}
For every integer $k\geq 2$, 
WTD($k$) is an infinite graph family.	
\end{corollary}

The \hp problem is the decision problem that takes as input two
Sperner hypergraphs $H$ and $H'$ and asks whether $H'$ is the transversal hypergraph $H^*$ of $H$. 

\begin{theorem}[\cite{eiter:1995}, \cite{boros:1998}]\label{thm:mintransk}
For every positive integer $k$,
the \hp problem is solvable in polynomial time if all hyperedges of one of the two
hypergraphs $H$ and $H'$ are of size at most $k$.
\end{theorem}

Theorem \ref{thm:mintransk} has the following consequence: 
\begin{corollary}\label{cor:wtdrec}
(\cite{gozupek:2017}) For every positive integer $k$,
the following problem is solvable in polynomial time: Given a Sperner hypergraph $H$, determine whether all minimal transversals of $H$ are of size $k$.
\end{corollary}

The complexity of recognition of WTD graphs with bounded total domination number can now be derived from Corollary \ref{cor:wtdrec}.

\begin{theorem}\label{thm:recWTDk}
For every positive integer $k$, the problem of recognizing  WTD($k$) graphs can be solved in polynomial time.
\end{theorem}
\begin{proof}
Let $G$ be graph with no isolated vertices and $\gamma_{t}(G)=k$. 
Consider the hypergraph $H_{G}=(V, \mathcal{E})$, where $\mathcal{E}$ contains the inclusion-minimal elements of $\{N(v): v \in V\}$. 
Observe that $H_{G}$ is Sperner and that the minimal transversals of $H_{G}$ are exactly the minimal total dominating sets of $G$ by Lemma \ref{lem:hg}. 
It follows that $G$ is WTD if and only if all minimal transversals of $H_{G}$ are of size $k$. 
By Corollary \ref{cor:wtdrec}, this condition can be tested in polynomial time.
\end{proof}

\section{WTD Graphs with Total Domination Number Two} \label{sec:WTDtwo}
In this section, we study WTD graphs whose total domination number is 2.
We give complete characterizations of WTD(2) graphs with packing number 2 and triangle-free WTD(2) graphs. We also show that planar WTD(2) graphs with minimum degree at least 3 have at most 16 vertices.

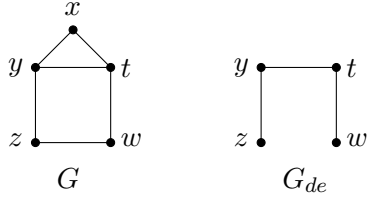
\begin{figure}\center
\begin{tikzpicture}[line cap=round,line join=round,>=triangle 45,x=1.0cm,y=1.0cm]
\clip(0.0,0.1) rectangle (6,3);
\draw  (1,1)-- (1,2);
\draw  (1,2)-- (2,2);
\draw  (2,1)-- (2,2);
\draw (1,1)-- (2,1);
\draw (1,2)-- (1.5,2.5);
\draw (2,2)-- (1.5,2.5);
\draw (1.15,0.8) node[anchor=north west] {$G$};
\begin{scriptsize}
\draw [fill=black] (1,1) circle (1.5pt);
\draw [fill=black] (1,2) circle (1.5pt);
\draw [fill=black] (2,1) circle (1.5pt);
\draw [fill=black] (2,2) circle (1.5pt);
\draw [fill=black] (1.5,2.5) circle (1.5pt);
\end{scriptsize}
\draw (1.25,3) node[anchor=north west] {$x$};
\draw (0.5,2.25) node[anchor=north west] {$y$};
\draw (0.5,1.25) node[anchor=north west] {$z$};
\draw (2,2.25) node[anchor=north west] {$t$};
\draw (2,1.25) node[anchor=north west] {$w$};

\begin{scope}[shift={(3,0)}]
\draw  (1,1)-- (1,2);
\draw (1,2)-- (2,2);
\draw (2,1)-- (2,2);
\draw (0.5,2.25) node[anchor=north west] {$y$};
\draw (0.5,1.25) node[anchor=north west] {$z$};
\draw (2,2.25) node[anchor=north west] {$t$};
\draw (2,1.25) node[anchor=north west] {$w$};
\draw (1.15,0.8) node[anchor=north west] {$G_{de}$};
\begin{scriptsize}
\draw [fill=black] (1,1) circle (1.5pt);
\draw [fill=black] (1,2) circle (1.5pt);
\draw [fill=black] (2,1) circle (1.5pt);
\draw [fill=black] (2,2) circle (1.5pt);
\end{scriptsize}
\end{scope}
\end{tikzpicture}
\caption{A WTD(2) graph $G$ and the graph $G_{de}$ obtained by the dominating edges of $G$.
}\label{fig:Gdeex}
\end{figure}

Let $G$ be a WTD(2) graph.
Note that every minimal TDS of $G$ is a pair consisting of endpoints of an edge of $G$. Consequently, every WTD(2) graph is connected.
We will call an edge of $G$ whose endpoints is a TDS of $G$ a \emph{dominating edge} of $G$.
Let $G_{de}$ be the graph with vertex set $\cup_{S \in MTDS(G)} S$ (i.e., vertices of $G$ serve as an endpoint of a dominating edge) and edge set which consists of dominating edges of $G$.
In other words,
$G_{de}$ is the edge-induced subgraph of $G$ obtained by the dominating edges.
See Figure \ref{fig:Gdeex} for an example.

\begin{remark}
Notice that the graph $G_{de}$ and the subgraph of $G$ induced by $V(G_{de})$ are not necessarily the same.
In general, $G_{de}$ is a subgraph of $G$ but not necessarily an induced subgraph of $G$ with respect to a set of vertices.
\end{remark}

A set $S$ is a \emph{vertex cover} of a graph $G$ if every edge of $G$ has an endpoint from $S$.
Let $MVC(G)$ denote the set of all minimal vertex covers of the graph $G$.

\begin{proposition}\label{prop:wtd2}
Let $G$ be a WTD(2) graph.
For every minimal vertex cover $S$ of $G_{de}$ there exists a vertex $v_S$ in $G$ such that $N(v_S)=S$.
\end{proposition}
\begin{proof}
Note that every minimal vertex cover $S$ of $G_{de}$ is a minimal transversal of $MTDS(G)$.
Therefore, by Proposition \ref{prop:transmtdsg} there exists a vertex in $G$ whose neighborhood is exactly $S$.
\end{proof}

\subsection{Characterization of WTD(2) Graphs with Packing Number 2}\label{sec:wtd2ro2}
A set $S\subseteq V(G)$ is called a packing of $G$ if $N[u]\cap N[v]=\emptyset$ for every distinct $u,v\in S$.
\emph{The packing number} $\rho(G)$ is the maximum size of a packing of $G$. It is well-known that for any graph $G$ we have $\rho(G) \leq \gamma(G) \leq \gamma_t(G)$.
Therefore, if $\gamma_t(G)=2$, then $\rho(G)$ is either 1 or 2.
In this subsection, we provide a characterization of WTD(2) graphs $G$ with $\rho(G)=2$.
In particular, this characterization allows us to construct any WTD(2) graph with $\rho(G)=2$.

Let $W_2$ be the set of graphs obtained as follows:\\
{\bf Step 1:} Choose a bipartite graph $H$ with no isolated vertices.\\
{\bf Step 2:} For every $S\in MVC(H)$, choose a new vertex $v_S$ and draw edges from $v_S$ to every vertex in $S$.\\
{\bf Step 3:} For each edge $uv$ in $H$ and every $w\in V(H)\backslash \{u,v\}$, add the edges $wu$ and/or $wv$ if needed to make sure $w$ is adjacent to at least one of $u$ and $v$..\\
{\bf Step 4:} 
Choose a new graph $H'$ (might be the empty graph) which is disjoint from the current graph. Then for each edge $uv$ in $H$ and every $w\in V(H')$, draw at least one of the edges $wu$ and $wv$.\\
A graph in $W_2$ is given in Figure \ref{fig:w2process}.
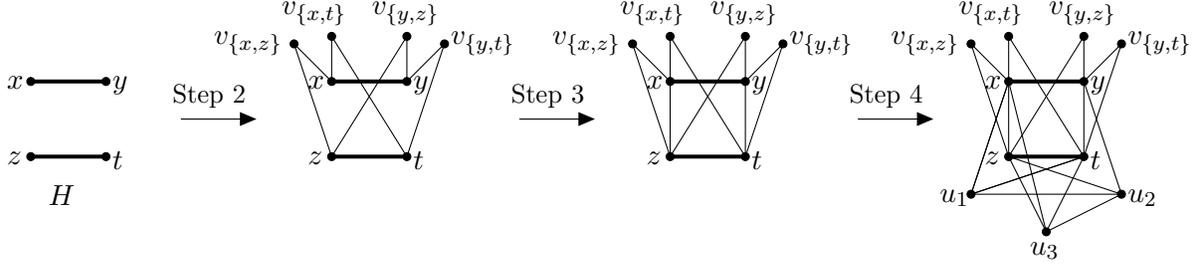
\begin{figure}\center
	\begin{tikzpicture}[line cap=round,line join=round,>=triangle 45,x=1.0cm,y=1.0cm]
	\clip(0.0,0.5) rectangle (16,5);
	\draw [line width=1.5pt] (0.5,3.)-- (1.5,3.);
	\draw [line width=1.5pt] (0.5,2.)-- (1.5,2.);
	\draw (0.6,1.75) node[anchor=north west] {$H$};
	\draw [->] (2.5,2.5) -- (3.5,2.5);
	\draw (2.25,3.1) node[anchor=north west] {{\small Step 2}};
	\draw [line width=1.5pt] (4.5,3.)-- (5.5,3.);
	\draw [line width=1.5pt] (4.5,2.)-- (5.5,2.);
	\draw (4.,3.5)-- (4.5,2.);
	\draw (4.,3.5)-- (4.5,3.);
	\draw (4.5,3.6)-- (4.5,3.);
	\draw (4.5,3.6)-- (5.5,2.);
	\draw (5.5,3.6)-- (4.5,2.);
	\draw (5.5,3.6)-- (5.5,3.);
	\draw (6.,3.5)-- (5.5,3.);
	\draw (6.,3.5)-- (5.5,2.);
	\draw (2.8,3.8) node[anchor=north west] {$v_{\{x,z\}}$};
	\draw (3.7,4.2) node[anchor=north west] {$v_{\{x,t\}}$};
	\draw (4.9,4.2) node[anchor=north west] {$v_{\{y,z\}}$};
	\draw (5.95,3.8) node[anchor=north west] {$v_{\{y,t\}}$};
	\draw (0.05,3.2) node[anchor=north west] {$x$};
	\draw (1.45,3.2) node[anchor=north west] {$y$};
	\draw (0.05,2.2) node[anchor=north west] {$z$};
	\draw (1.45,2.2) node[anchor=north west] {$t$};
	\draw (4.05,3.2) node[anchor=north west] {$x$};
	\draw (5.45,3.2) node[anchor=north west] {$y$};
	\draw (4.05,2.2) node[anchor=north west] {$z$};
	\draw (5.45,2.2) node[anchor=north west] {$t$};
	\begin{scriptsize}
	\draw [fill=black] (0.5,3.) circle (1.5pt);
	\draw [fill=black] (1.5,3.) circle (1.5pt);
	\draw [fill=black] (0.5,2.) circle (1.5pt);
	\draw [fill=black] (1.5,2.) circle (1.5pt);
	\draw [fill=black] (4.5,3.) circle (1.5pt);
	\draw [fill=black] (5.5,3.) circle (1.5pt);
	\draw [fill=black] (4.5,2.) circle (1.5pt);
	\draw [fill=black] (5.5,2.) circle (1.5pt);
	\draw [fill=black] (4.,3.5) circle (1.5pt);
	\draw [fill=black] (4.5,3.6) circle (1.5pt);
	\draw [fill=black] (5.5,3.6) circle (1.5pt);
	\draw [fill=black] (6.,3.5) circle (1.5pt);
	\end{scriptsize}
	
	\begin{scope}[shift={(4.5,0)}]
	\draw [->] (2.5,2.5) -- (3.5,2.5);
	\draw (2.25,3.1) node[anchor=north west] {{\small Step 3}};
	\draw [line width=1.5pt] (4.5,3.)-- (5.5,3.);
	\draw [line width=1.5pt] (4.5,2.)-- (5.5,2.);
	\draw (4.5,3)-- (4.5,2);
	\draw (5.5,3)-- (5.5,2);
	\draw (4.,3.5)-- (4.5,2.);
	\draw (4.,3.5)-- (4.5,3.);
	\draw (4.5,3.6)-- (4.5,3.);
	\draw (4.5,3.6)-- (5.5,2.);
	\draw (5.5,3.6)-- (4.5,2.);
	\draw (5.5,3.6)-- (5.5,3.);
	\draw (6.,3.5)-- (5.5,3.);
	\draw (6.,3.5)-- (5.5,2.);
	\draw (2.8,3.8) node[anchor=north west] {$v_{\{x,z\}}$};
	\draw (3.7,4.2) node[anchor=north west] {$v_{\{x,t\}}$};
	\draw (4.9,4.2) node[anchor=north west] {$v_{\{y,z\}}$};
	\draw (5.95,3.8) node[anchor=north west] {$v_{\{y,t\}}$};
	\draw (4.05,3.2) node[anchor=north west] {$x$};
	\draw (5.45,3.2) node[anchor=north west] {$y$};
	\draw (4.05,2.2) node[anchor=north west] {$z$};
	\draw (5.45,2.2) node[anchor=north west] {$t$};
	\begin{scriptsize}
	\draw [fill=black] (4.5,3.) circle (1.5pt);
	\draw [fill=black] (5.5,3.) circle (1.5pt);
	\draw [fill=black] (4.5,2.) circle (1.5pt);
	\draw [fill=black] (5.5,2.) circle (1.5pt);
	\draw [fill=black] (4.,3.5) circle (1.5pt);
	\draw [fill=black] (4.5,3.6) circle (1.5pt);
	\draw [fill=black] (5.5,3.6) circle (1.5pt);
	\draw [fill=black] (6.,3.5) circle (1.5pt);
	\end{scriptsize}
	\end{scope}
	
	\begin{scope}[shift={(9,0)}]
	\draw [->] (2.5,2.5) -- (3.5,2.5);
	\draw (2.25,3.1) node[anchor=north west] {{\small Step 4}};
	\draw [line width=1.5pt] (4.5,3.)-- (5.5,3.);
	\draw [line width=1.5pt] (4.5,2.)-- (5.5,2.);
	\draw (4.5,3)-- (4.5,2);
	\draw (5.5,3)-- (5.5,2);
	\draw (4.,3.5)-- (4.5,2.);
	\draw (4.,3.5)-- (4.5,3.);
	\draw (4.5,3.6)-- (4.5,3.);
	\draw (4.5,3.6)-- (5.5,2.);
	\draw (5.5,3.6)-- (4.5,2.);
	\draw (5.5,3.6)-- (5.5,3.);
	\draw (6.,3.5)-- (5.5,3.);
	\draw (6.,3.5)-- (5.5,2.);
	\draw (2.8,3.8) node[anchor=north west] {$v_{\{x,z\}}$};
	\draw (3.7,4.2) node[anchor=north west] {$v_{\{x,t\}}$};
	\draw (4.9,4.2) node[anchor=north west] {$v_{\{y,z\}}$};
	\draw (5.95,3.8) node[anchor=north west] {$v_{\{y,t\}}$};
	\draw (4.05,3.2) node[anchor=north west] {$x$};
	\draw (5.45,3.2) node[anchor=north west] {$y$};
	\draw (4.05,2.2) node[anchor=north west] {$z$};
	\draw (5.45,2.2) node[anchor=north west] {$t$};
	\begin{scriptsize}
	\draw [fill=black] (4.5,3.) circle (1.5pt);
	\draw [fill=black] (5.5,3.) circle (1.5pt);
	\draw [fill=black] (4.5,2.) circle (1.5pt);
	\draw [fill=black] (5.5,2.) circle (1.5pt);
	\draw [fill=black] (4.,3.5) circle (1.5pt);
	\draw [fill=black] (4.5,3.6) circle (1.5pt);
	\draw [fill=black] (5.5,3.6) circle (1.5pt);
	\draw [fill=black] (6.,3.5) circle (1.5pt);
	\draw [fill=black] (4.,1.5) circle (1.5pt);
	\draw [fill=black] (5.,1.) circle (1.5pt);
	\draw [fill=black] (6.,1.5) circle (1.5pt);
	\end{scriptsize}
	\draw (4,1.5)-- (4.5,3.);
	\draw (4,1.5)-- (5.5,2.);
	\draw (5,1)-- (4.5,3.);
	\draw (5,1)-- (4.5,2.);
	\draw (5,1)-- (5.5,2.);
	\draw (6,1.5)-- (4.5,2.);
	\draw (6,1.5)-- (5.5,3.);
	\draw (4,1.5)-- (4.5,3.);
	\draw (4,1.5)-- (5.5,2.);
	\draw (4,1.5)-- (6,1.5);
	\draw (6,1.5)-- (5,1);
	\draw (3.45,1.7) node[anchor=north west] {$u_1$};
	\draw (5.95,1.7) node[anchor=north west] {$u_2$};
	\draw (4.65,1.) node[anchor=north west] {$u_3$};
	\end{scope}
	\end{tikzpicture}
	\caption{A graph in $W_2$ obtained by the given process.
		Bold edges represent the dominating edges.
	}\label{fig:w2process}
\end{figure}

\begin{lemma}\label{thm:w2arewtd2}
If a graph $G$ is in $W_2$, then $G$ is a WTD(2) graph with $\rho(G)=2$.
\end{lemma}
\begin{proof}
Let $G\in W_2$ and $H=(U,V,E)$ be the bipartite graph in the first step of the construction of $G$.
We first show that the packing number of $G$ is 2.
As $H$ has no isolated vertices, both $U$ and $V$ are minimal vertex covers of $H$.
Thus,
the vertices $v_U$ and $v_V$ have disjoint closed neighborhoods since $N(v_U)=U$ and $N(v_V)=V$ and hence, we get $\rho(G)\geq 2$. 
Clearly, by construction, every edge of $H$ is a dominating edge of $G$.
Therefore, we get $\gamma_t(G)=2$.
Since $\rho(G)\leq \gamma_t(G)$, we obtain $\rho (G) \leq 2$ and hence, $\rho(G)=2$.
	
Now let $T$ be a minimal TDS of $G$ other than the edges of $H$.
Then $T$ contains at most one endpoint of an edge of $H$ because otherwise $T$ contains a TDS, which contradicts with $T$ being minimal.
Therefore, $V(H)\backslash T$ is a vertex cover of $H$ and hence, it contains a minimal vertex cover $S$ of $H$.
By construction there exists a vertex $v_S$ with $N(v_S)=S$.
As $S\subseteq V(H)\backslash T$, we obtain $N(v_S)\cap T=\emptyset$, which contradicts with $T$ being a TDS of $G$.
Consequently, edges of $H$ are the only minimal TDSs of $G$ and hence,
$G$ is a WTD(2) graph and $G_{de}=H$.
\end{proof}

\begin{lemma}\label{lem:wtdimpw2}
Let $G$ be a WTD(2) graph with $\rho (G)=2$. Then, $G$ is in $W_2$.
\end{lemma}
\begin{proof}
Let $\{x,y\}$ be a packing with minimum $|N[x]|+ |N[y]|$.
Note that every dominating edge of $G$ has one endpoint from $N(x)$ and one from $N(y)$ and hence,
$G_{de}$ is a bipartite graph, say with parts $X$ and $Y$ where $X\subseteq N(x)$ and $Y\subseteq N(y)$.
	
We next show that $X=N(x)$ and $Y=N(y)$.
By symmetry, it suffices to prove $X=N(x)$.
Notice that $G_{de}$ has no isolated vertices and therefore,
$X$ is a minimal vertex cover of $G_{de}$.
By Proposition \ref{prop:wtd2} there exists a vertex $v_X$ satisfying $N(v_X)=X$.
Suppose that $X\neq N(x)$.
Then, we get $X \subset N(x)$.
Clearly $v_X \neq y$.
Moreover, $v_X \notin N(y)$ since $y\notin X=N(v_X)$.
Thus, we get $N[v_X] \cap N[y] = \emptyset$ and hence $\{v_X,y\}$ is a packing of $G$.
However, we obtain $|N[v_X]|+| N[y]|< |N[x]|+ |N[y]|$ since $X \subset N(x)$, which yields a contradiction with the definition of the packing $\{x,y\}$.
Consequently, we get $X=N(x)$ and hence, we may take $v_X=x$.
Similarly, we have $Y=N(y)$ and we may assume $v_Y=y$.
	
Now let $S$ be a minimal vertex cover of $G_{de}$.
By Proposition \ref{prop:wtd2} there exists a vertex $v_S$ satisfying $N(v_S)=S$.	If $S=X$ or $S=Y$, we can take $v_S$ to be $x$ or $y$, respectively, and in both cases, we have $v_S\notin V(G_{de})$.
Otherwise, suppose that $v_S \in V(G_{de})=X\cup Y$.
Without loss of generality, let $v_S \in X$.
Then, as $X=N(x)$, we get $x\in N(v_S)=S \subseteq N(x) \cup N(y)$, contradiction.
Therefore, $v_S$ is not a vertex of $G_{de}$, that is, $v_S \in V(G)\backslash V(G_{de})$.

Finally, we see that one can obtain the graph $G$ by following the procedure in the definition of $W_2$ with the initial bipartite graph $H=G_{de}$.
\end{proof}

Combining the results in Lemma \ref{thm:w2arewtd2} and Lemma \ref{lem:wtdimpw2} gives the following structural characterization of WTD(2) graphs with $\rho(G)=2$. 
Moreover, by definition of the class $W_2$, this provides us with a procedure to construct any WTD(2) graph  with $\rho(G)=2$.
\begin{theorem}\label{cor:w2wtd2}
A graph $G$ is WTD(2) with $\rho(G)=2$ if and only if $G\in W_2$.
\end{theorem}

Given a graph $G$, the \textit{diameter} of $G$, denoted by $diam(G)$ is the maximum length of a shortest path between any pair of vertices of $G$.
Let $G$ be a graph such that $\gamma_t(G)=2$.
Then, it is easy to see that $diam(G) \leq 3$.
Moreover, whenever $\gamma_t(G)=2$, we have $diam(G)=3$ if and only if $\rho(G)=2$ and therefore,
in all the statements in Lemma \ref{thm:w2arewtd2}, Lemma \ref{lem:wtdimpw2} and Theorem \ref{cor:w2wtd2},
the condition $\rho(G)=2$ can be replaced with $diam(G)=3$.
\begin{corollary}
A graph $G$ is WTD(2) with $diam(G)=3$ if and only if $G\in W_2$.
\end{corollary}


One may attempt to modify the description of $W_2$ graphs in order to describe all WTD(2) graphs with $\rho(G)=1$.
In the first step of the process of building a graph in $W_2$,
if one starts with a non-bipartite graph $H$ with no isolated vertices,
then the resulting graph is still WTD(2) but has packing number 1.
However, not every WTD(2) graph $G$ with $\rho(G)=1$ can be obtained in this way.
For example, consider the graph presented in Figure \ref{fig:Gdeex}.
To obtain this graph $G$, 
in Step 1 one should definitely choose $H$ to be the graph with vertex set $\{z,y,t,w\}$ and edge set $\{zy,yt,tw\}$ which is indeed $G_{de}$.
However, in Step 2 if one chooses a new vertex $v_S$ for $S=\{y,w\}$ (which  is a minimal vertex cover of $G_{de}$), then the graph $G$ can not be obtained. So, the complete characterization of WTD(2) graphs with  $\rho(G)=1$ is left as an open question.

\subsection{Triangle-free WTD(2) Graphs} \label{sec:bipWTD}
In this subsection, we provide characterization of triangle-free WTD(2) graphs.

\begin{lemma}\label{obs:bipgammat2}
If $G$ is a triangle-free graph with $\gamma_t(G)=2$,
then $G$ is a bipartite graph and we have
$$\rho(G)=
\begin{cases}
1, & \text{ if } $G$ \text{ is complete bipartite}\\
2, & \text{ otherwise}
\end{cases}
$$
\end{lemma}
\begin{proof}
Let $uv$ be a dominating edge of $G$.
Then we have $N(u)\cup N(v)=V(G)$.
As $G$ is triangle-free, none of two adjacent vertices have a common neighbor.
Therefore, we have $N(u)\cap N(v)=\emptyset$ and also see that both $N(u)$ and $N(v)$ are independent sets.
We consequently obtain that $G$ is a bipartite graph with parts $N(u)$ and $N(v)$.
Since $\rho(G)\leq \gamma_t(G)=2$, we have $\rho(G)\in \{1,2\}$.
Moreover, it is clear that $\rho(G)=1$ if and only if each vertex in $N(u)$ is adjacent to all the vertices in $N(v)$, i.e., $G$ is a complete bipartite graph.
\end{proof}

For a bipartite graph with parts $X$ and $Y$,
define $X_u=\{x\in X: N(x)=Y\}$ and $Y_u=\{y\in Y: N(y)=X\}$.
In other words,
$X_u$ (resp. $Y_u$) is the set of vertices in $X$ (resp. $Y$) which are adjacent to every vertex in $Y$ (resp. $X$).
The following result characterizes all triangle-free WTD(2) graphs. 

\begin{theorem}\label{thm:bipwtd2}
The following three statements are equivalent:\\
(i) $G$ is a triangle-free WTD(2) graph.\\
(ii) $G$ is a bipartite WTD(2) graph.\\
(iii) $G$ is complete bipartite graph or $G$ is a bipartite graph with parts $X$ and $Y$ such that there exist vertices $a\in X\backslash X_u$ and $b\in Y\backslash Y_u$ satisfying $N(a)=Y_u \neq \emptyset$ and $N(b)=X_u\neq \emptyset$.
\end{theorem}
\begin{proof}
By Lemma \ref{obs:bipgammat2} we see that (i) implies (ii).
On the other hand, the implication (iii)$\rightarrow$(i) can be easily verified and hence,
the proof finishes if we show that (ii) implies (iii).
Now let $G$ be a bipartite WTD(2) graph, say with parts $X$ and $Y$.
Clearly we will only consider the case when $G$ is not a complete bipartite graph.
By definition of $X_u$ and $Y_u$, note that every dominating edge of $G$ has one endpoint in $X_u\neq \emptyset$ and one endpoint in $Y_u\neq \emptyset$.
Moreover, any edge $xy$ where $x\in X_u$ and $y\in Y_u$ is a dominating edge of $G$.
Therefore, $G_{de}$ is the subgraph of $G$ induced by $X_u \cup Y_u$ and it is complete bipartite.
Thus, $G_{de}$ has only two minimal vertex covers, namely $X_u$ and $Y_u$.
Then, definition of a graph in $W_2$ and Theorem \ref{cor:w2wtd2} imply the existence of the vertices $a\in X\backslash X_u$ and $b\in Y\backslash Y_u$ with $N(a)=Y_u$ and $N(b)=X_u$.
\end{proof}


Although a polynomial time recognition algorithm for WTD(2) graphs follows from Theorem \ref{thm:recWTDk}, the characterization in Theorem \ref{thm:bipwtd2} provides us with a simple linear time recognition algorithm.

\begin{corollary}
Triangle-free WTD(2) graphs can be recognized in linear time.
\end{corollary}
\begin{proof}
Given a graph $G$, one can check whether it is a connected bipartite graph and if so, find its unique bipartition $(X,Y)$ in linear time (in the number of vertices and edges of $G$). Then, sets $X_u$ and $Y_u$ can be identified simply by assigning every vertex $x\in X$ such that $d(x)=|Y|$ into $X_u$, and $y\in Y$ such that $d(y)=|X|$ into $Y_u$. According to Theorem  \ref{thm:bipwtd2}, $G$ is triangle-free WTD(2) if and only if either $X_u=X$ and $Y_u=Y$ (thus, $G$ is complete bipartite), or the removal of $X_u$ and $Y_u$ leaves at least one isolated vertex in each one of $X$ and $Y$. Clearly, all these checks take only linear time. 
\end{proof}

\subsection{Planar WTD(2) Graphs} \label{sec:planarWTD}
In this subsection,
we study planar WTD(2) graphs whose minimum degree is at least three
and show that such graphs can have at most sixteen vertices.
Throughout this section,
we frequently use the fact that a graph obtained by an edge contraction of a planar graph is also planar.
Recall also that a planar graph contains no $K_5$ or $K_{3,3}$.

\begin{observation}
Let $G$ be a WTD(2) graph.
The vertex obtained by edge contraction of a dominating edge is a universal vertex in the new graph.
\end{observation}

Let $\nu(G)$ denote the matching number of a graph $G$.

\begin{lemma}\label{lem:plwrd2vg8}
Let $G$ be a planar WTD(2) graph.
If $\nu(G_{de})\geq 3$, then $|V(G)|\leq 8$.
\end{lemma}
\begin{proof}
Suppose that $\nu(G_{de})\geq 3$ and $G$ has at least 9 vertices.
Then, $G$ has three independent dominating edges, say $u_1v_1,u_2v_2$ and $u_3v_3$, and three vertices other than $u_1,u_2,u_3,v_1,v_2,v_3$, say $w_1,w_2$ and $w_3$.
Now contract the edges $u_1v_1,u_2v_2$ and $u_3v_3$.
In the resulting graph, new three vertices and $w_1,w_2,w_3$ contain a $K_{3,3}$, which contradicts with the planarity.
\end{proof}

\begin{lemma}\label{lem:wtd2del3nu2}
If $G$ is a WTD(2) graph with $\delta(G)\geq 3$, then $\nu (G_{de})\geq 2$.
\end{lemma}
\begin{proof}
Let $G$ be a WTD(2) graph with $\delta(G)\geq 3$.
It suffices to show that $G$ has two independent dominating edges.
Let $xy$ be a dominating edge of $G$.
Since the minimum degree is at least three, each vertex of $G$ has at least one neighbor in $V(G)\setminus \{x,y\}$.
Therefore, $V(G)\setminus \{x,y\}$ is a TDS of $G$ and hence,
it contains a dominating edge $ab$ since $G$ is WTD(2).
As the dominating edges $xy$ and $ab$ share no vertex, we get $\nu (G_{de})\geq 2$.
\end{proof}

Combining the results in Lemmas \ref{lem:plwrd2vg8} and  \ref{lem:wtd2del3nu2} gives the following result.
\begin{proposition}\label{prop:nu2n8}
If $G$ is a planar WTD(2) graph with $\delta(G)\geq 3$,
then $\nu(G_{de})=2$ or $|V(G)|\leq 8$.
\end{proposition}

We next study planar WTD(2) graphs
whose minimum degree is at least 3 and matching number is 2.

\begin{proposition}\label{prop:vGde=2}
If $G$ is a planar WTD(2) graph with $\delta(G) \geq 3$ and $\nu(G_{de})=2$,
then $|V(G)|\leq 16$.
\end{proposition}
\begin{proof}
Let $ab$ and $xy$ be two independent dominating edges of $G$ and $H=G-\{a,b,x,y\}$.
Let $H_1,\dots,H_m$ be the connected components of $H$ and order of $H_i$ be $h_i$ for $i=1,\dots,m$.
Note that it suffices to show that $h_1+\cdots +h_m \leq 12$.

We first prove that each $H_i$ is a path or a singleton.
Note that it suffices to show that maximum degree of $H$ is at most 2 and $H$ contains no cycle.
Suppose that a vertex $v$ of $H$ has three neighbors, say $v_1,v_2,v_3$, in $H$.
Then contraction of the edges $ab$ and $xy$ gives rise to a $K_{3,3}$ with parts $\{ab, xy,v\}$ and $\{v_1,v_2,v_3\}$, contradiction.
Therefore, every vertex in $H$ has at most two neighbors in $H$.
Suppose that $H$ has a cycle, say $v_1,v_2,\dots, v_k$.
Contract the edge $v_kv_{k-1}$ and denote the new point by $v_{k-1}$.
Then contract the edge $v_{k-1}v_{k-2}$ and denote the new point by $v_{k-2}$ and so on.
Follow this process until we get a triangle $v_1,v_2,v_3$.
Then contracting the edges $ab$ and $xy$ yields a $K_5$ with vertices $ab,xy,v_1,v_2,v_3$, contradiction.
Thus, $H$ has no cycle and hence, $H$ is a disjoint union of paths and singletons.

We next show that for every vertex $u\in H$ we have $|N(u)\cap \{a,b,x,y\}|\geq 3$ or $|(N(u)\cup N(v))\cap \{a,b,x,y\}|\geq 3$ for some neighbor $v\in V(H)$ of $u$.
Since both $ab$ and $xy$ are dominating edges, the intersection $N(u)\cap \{a,b,x,y\}$ has at least two elements: one from $\{a,b\}$ and one from $\{x,y\}$.
Consider the case when $|N(u)\cap \{a,b,x,y\}|= 2$.
Without loss of generality, let $N(u)\cap \{a,b,x,y\}=\{a,x\}$.
Since the minimum degree of $G$ is at least 3, there is no vertex $v \in G$ such that $N(v)=\{a, x\}$. Hence, by Proposition \ref{prop:wtd2} the set $\{a,x\}$ is not a vertex cover of $G_{de}$.
Then, there exists an edge $wv$ of $G_{de}$ such that $\{w,v\} \cap \{a,x\} =\emptyset$.
Thus, as $\nu(G_{de})=2$ and $ab,xy \in G_{de}$, we have $wv=by$ or $w\in \{b,y\}$ and $v\in V(H)$.
Recall that $wv$ is a dominating edge in $G$ and hence,
$u$ is adjacent to $w$ or $v$.
Therefore, the case $wv=by$ is impossible and we see that $v$ is adjacent to $u$.
Consequently, we get $|(N(u)\cup N(v))\cap \{a,b,x,y\}|\geq 3$ since $w\in \{b,y\}$ is a neighbor of $v$.
Note that this result implies that if $u$ is a singleton, then it has at least three neighbors among $a,b,x,y$; otherwise, contraction of the edge $uv$ gives rise to a vertex adjacent to at least three of $a,b,x,y$.

We then apply the following process for each $i=1,\dots,m$:
If $h_i\leq 3$, contract the edges of $H_i$ and obtain a singleton.
If $h_i\geq 4$, let $H_i$ be the path $v_1,v_2,\dots,v_k$ where $k=h_i$.
First, contract $v_1v_2$ and $v_{k-1}v_k$.
Then contract the paths $v_3v_4v_5, v_6v_7v_8,\dots$ and so on.
Note that for every $i$ we obtain at least $2+\lfloor (h_i-4)/3\rfloor=\lfloor (h_i+2)/3\rfloor$ vertices adjacent to at least three of $a,b,x,y$.
Therefore, each such vertex is adjacent to both $a$ and $b$ or adjacent to both $x$ and $y$.
Assume that the number of vertices having at least three neighbors among $a,b,x,y$ in the resulting graph is more than 4.
Then, by pigeonhole principle, there will be three distinct vertices $u_1,u_2$ and $u_3$ each of which is adjacent, without loss of generality, to both $a$ and $b$. 
Then, contraction of the edge $xy$ gives a $K_{3,3}$ with parts $\{a,b,xy\}$ and $\{u_1,u_2,u_3\}$, contradicting with the planarity of $G$.
Thus, there are at most 4 vertices once the contraction process is terminated, that is, $\sum_{i=1}^m \lfloor (h_i+2)/3\rfloor \leq 4$.
Since $h_i$ is an integer, the inequality $h_i/3 \leq \lfloor (h_i+2)/3\rfloor$ holds, implying that $\sum_{i=1}^m h_i/3  \leq 4$ which yields $\sum_{i=1}^m h_i\leq 12$ as desired.
\end{proof} 

Propositions \ref{prop:nu2n8} and \ref{prop:vGde=2} imply that, unlike the general case stated in Corollary \ref{cor:infwtdk}, there is a finite number of planar WTD(2) graphs with $\delta(G)\geq 3$.

\begin{theorem}\label{thm:planar}
If $G$ is a planar WTD(2) graph with $\delta(G) \geq 3$, then $|V(G)|\leq 16$.
\end{theorem}

In contrast, there is no upper bound on the number of vertices for planar WTD(2) graphs with minimum degree 1 or 2.
For example, consider a star with arbitrarily many leaves and a graph with arbitrarily many triangles sharing a common edge, respectively.

\section{Girth of WTD Graphs} \label{sec:girthWTD}
In this section, we provide a relation between the minimum degree and the girth for WTD graphs.
We show that if the minimum degree is more than two in a WTD graph, then the graph contains a cycle of length at most twelve.
It is shown in \cite{HaRa97} that if $G$ is a WTD graph with $\delta(G)\geq 2$,
then the girth of $G$, $g(G)$, is at most 14.
\begin{theorem}[Theorem 4.1 in \cite{HaRa97}]
Suppose $G$ is a connected graph with no leaves such that $G$ has girth at least fifteen.
Then $\gamma_t(G)<\Gamma_t(G)$.
\end{theorem}

By following the idea in the proof of Theorem 4.1 in \cite{HaRa97},
one can find other relations between $\delta(G)$ and $g(G)$ of a WTD graph $G$.
Before presenting such extensions, we need the following useful lemma, which is also given in \cite{HaRa97}:
\begin{lemma}\label{lem:wtdreduction}
Let $G$ be a WTD graph, $u_1v_1,\dots, u_mv_m$ be a subset of the edges of $G$ and $A=\cup_{i=1}^m \{u_i,v_i\}$.
If the subgraph of $G$ induced by $A$ is disjoint union of $m$ $K_2$s and $G-N[A]$ has no isolated vertices,
then $G-N[A]$ is also WTD.
\end{lemma}
\begin{proof}
Let $S$ be a minimal TDS of $G-N[A]$.
We claim that $S\cup A$ is a minimal TDS of $G$.
It is easy to see that it is a TDS of $G$.
Suppose that $S\cup A$ contains another TDS of $G$, say $T$.
Then $T\cap S$ is a TDS of $G-N[A]$ and hence, since $S$ is minimal we get $T\cap S=S$.
Therefore, we obtain $T=S\cup A'$ where $A'\subseteq A$.
If $A\backslash A'$ is nonempty, then without loss of generality we assume that $u_1 \in A\backslash A'$.
But then, $v_1$ is not dominated by $T$, contradiction.
Therefore, we have $A'=A$, which implies that $T=S\cup A$, that is, $S\cup A$ is minimal.

As every minimal TDS of $G$ has the same size, $|S|+m$ is independent of $S$ and hence,
$G-N[A]$ is a WTD graph as well.
\end{proof}

\begin{theorem}\label{thm:wtdmindeg3}
If $G$ is a WTD graph with $\delta (G)\geq 3$,
then $g(G)\leq 12$.
\end{theorem}
\begin{proof}
Assume that $G$ is a WTD graph with $\delta (G)\geq 3$ and $g(G)\geq 13$.
Let $P=v_1,v_2,v_3,v_4,v_5$ be a path in $G$.
For any vertex $v$ in $G$, let $d_P(v)=\min_{1\leq i \leq 5} dist(v,v_i)$.
Define $N_k$ to be the set of vertices $v$ with $d_P(v)=k$ for $k=1,2,\dots .$

First note that every vertex in $N_k$ has a neighbor in $N_{k-1}$ for every $k\geq 2$.
Moreover, for $k=1,2,3$, $N_k$ is an independent set since otherwise we obtain a cycle of length at most 11.
We will now show that for $k=1,\dots,4$, any vertex in $N_k$ has at least one neighbor in $N_{k+1}$. 
Suppose that there exist $k\leq 4$ and $v\in N_k$ such that $v$ is adjacent to no vertex in $N_{k+1}$.
By definition, it is clear that $v$ has no neighbor in $N_l$ for any $l\geq k+2$.
Therefore, all the neighbors of $v$ are in $\cup_{1\leq i \leq k} N_i$.
Thus, as $v$ has at least three neighbors, there exist three paths from $v$ to $P$ such that one of them has length $k$ and two of them have length at most $k+1$.
By a simple case analysis, considering the vertices of these paths on $P$ gives that there exist a cycle of length at most $2k+3\leq 11$, contradiction.

Now, let $N_2=\{w_1,\dots ,w_m\}$.
For every $i=1,\dots,m$, choose a neighbor of $w_i$ in $N_3$, say $u_i$.
Let $A=\cup_{i=1}^m \{w_i,u_i\}$.
For any $i\neq j$, $w_i$ is not adjacent to $u_j$ because otherwise we obtain a cycle of length at most 10.
Therefore, the induced subgraph of $G$ induced by $A$ is a disjoint union of $m$ $K_2$s.

Next, consider the graph $H=G-N[A]$.
Note that $N[A]$ consists of $N_1,N_2,N_3$ and some vertices in $N_4$.
Therefore, $P$ is a connected component of $H$.
As any vertex in $N_4$ has a neighbor in $N_5$, no vertex $v \in N_4 \cap V(H)$ is isolated in $H$.
Clearly, no vertex in $N_k$ with $k\geq 6$ is isolated in $H$ since it has a neighbor in $N_{k-1}$.
Suppose to the contrary that a vertex $v$ in $N_5$ is isolated in $H$.
Then $v$ has no neighbor in $N_5$ and $N_6$, and thus,
all its neighbors are in $N_4$.
Therefore, since there exist three paths from $v$ to $P$, this yields a cycle of length at most 12, contradiction. Consequently, $H$ has no isolated vertices and we can apply Lemma \ref{lem:wtdreduction} and conclude that $H$ is a WTD graph. However, $P$ is a component of $H$ and hence, it should be WTD as well.
Nevertheless, a path of length 4 is not a WTD graph (both $\{v_1,v_2,v_4,v_5\}$ and $\{v_2,v_3,v_4\}$ are minimal TDSs of $P$), contradiction.
\end{proof}

\section{Conclusion}
In this work, we studied graphs whose all minimal total dominating sets have the same size, a.k.a. well-totally-dominated graphs. We proved that well-totally-dominated graphs with bounded total domination number can be recognized in polynomial time. We then analyzed well-totally-dominated graphs with total domination number two for the special cases of triangle-free graphs and planar graphs. Finally, we focused on the girth of well-totally-dominated graphs. In particular, we proved that a well-totally dominated graph with minimum degree at least three has girth at most 12. We now conclude with several future research directions.

Although we proved in this paper that the problem of recognizing well-totally-dominated graphs with bounded total domination number can be solved in polynomial time, the complexity of the general case is an open research problem. Hence, we pose the following question:\\

\emph{Problem-1: What is the computational complexity of recognizing well-totally-dominated graphs?}\\

We have already characterized WTD(2) graphs with packing number $\rho(G)=2$ in Theorem \ref{cor:w2wtd2}. Since WTD(2) graphs have $\rho(G)\leq 2$, in order to complete the characterization of all WTD(2) graphs, it remains to answer the following question:\\

\emph{Problem-2: What are WTD(2) graphs  with  $\rho(G)=1$?}\\

Along the same line, one may consider to generalize our result in Theorem \ref{cor:w2wtd2}. It is well known that $\rho(G)\leq \gamma_t(G)\leq \Gamma_t(G)$; hence graphs with $\rho(G)=\Gamma_t(G)$ form a subclass of WTD graphs. This suggests our next open problem:\\

\emph{Problem-3: What are WTD(k) graphs  with  $\rho(G)=k$?}\\

Lastly, we have shown in Theorem \ref{thm:planar} that planar WTD(2) graphs with $\delta(G) \geq 3$ have at most $16$ vertices. Our intuition is that 16 is not a tight bound. Thus, we pose the following question:\\
 
\emph{Problem-4: Is the upper bound of 16 for the number of vertices of a planar WTD(2) graph with $\delta(G) \geq 3$ tight? Can we determine all (finitely many) planar WTD(2) graphs?}\\

\section*{Acknowledgments}
This work has been supported by the Scientific and Technological Research
Council of Turkey (TUBITAK) under grant no. 118E799. The work of Didem Gözüpek was also supported by the BAGEP Award of the Science Academy of Turkey.


\end{document}